\documentclass[12pt,reqno,a4paper]{amsart}
\usepackage{
    amsmath,  amsfonts, amssymb,  amsthm,   amscd,
    gensymb,  graphicx, comment,  etoolbox, url,
    booktabs, stackrel, mathtools,enumitem, mathdots,  microtype, lmodern,    mathrsfs, graphicx, tikz,  longtable,tabularx, float, tikz, pst-node, tikz-cd, multirow, tabularx, amscd,  bm, array, makecell, diagbox, booktabs,ragged2e, caption, subcaption }
\usepackage{makecell,slashbox}
\usepackage{xcolor}
\usepackage[all]{xy}
\usepackage{graphicx}

\usepackage[utf8]{inputenc}
\usepackage{microtype, fullpage, wrapfig,textcomp,mathrsfs,csquotes,fbb}
\usepackage[colorlinks=true, linkcolor=blue, citecolor=blue, urlcolor=blue, breaklinks=true]{hyperref}
\usepackage[capitalise]{cleveref}
\usepackage{todonotes}
\usetikzlibrary{positioning}
\usetikzlibrary{shapes,arrows.meta,calc}
\usetikzlibrary{arrows}

\newcommand{\RomanNumeralCaps}[1]
    {\MakeUppercase{\romannumeral #1}}
\newtheorem{theorem}{Theorem}[section]

\newtheorem{corollary}[theorem] {Corollary}
\newtheorem{definition}[theorem]{Definition}
\newtheorem{example}[theorem]{Example}
\newtheorem{lemma}[theorem]{Lemma}

\newtheorem{proposition}[theorem]{Proposition}
\newtheorem{remark}[theorem]{Remark}

\newcommand{\TC}{\mathrm{TC}}
\newcommand{\D}{D}

\newcommand{\ct}{\mathrm{cat}}

\newcommand{\sct}{\mathrm{secat}}

\newcolumntype{x}[1]{>{\centering\arraybackslash}p{#1}}

\begin{document}
\title[]{Equivariant homotopic distance}

\author[N. Daundkar]{Navnath Daundkar}
\address{Department of Mathematics, Indian Institute of Science Education and Research Pune, India.}
\email{navnathdaundkar23@gmail.com \\ navnath.daundkar@acads.iiserpune.ac.in}
\author[]{J.M. Garc\'ia-Calcines}
\address{Departamento de Matem\'aticas, Estad\'istica e Investigaci\'on Operativa, Universidad de La Laguna, Avenida Astrof\'isico Francisco S\'anchez S/N, 38200 La Laguna, Spain.}
\email{jmgarcal@ull.edu.es}

\thanks{}

\begin{abstract} 
We introduce and study the notion of \emph{equivariant homotopic distance} $D_G(f,g)$ between $G$-maps $f,g \colon X \to Y$. We show that the equivariant Lusternik--Schnirelmann category and the equivariant topological complexity are particular cases of this notion. This invariant also connects naturally with the equivariant sectional category. What makes $D_G$ distinctive, however, is that it provides a flexible framework centered on pairs of maps, within which one can derive results that are not immediate from the general setting.

In particular, we establish its basic properties, including homotopy invariance and a categorical proof of the triangle inequality valid in the equivariant context. We also obtain cohomological and dimension--connectivity bounds, and analyze structural applications to Hopf $G$--spaces and equivariant fibrations.
\end{abstract}

\maketitle

\section{Introduction}
The study of numerical homotopy invariants has long provided insight into the qualitative complexity of spaces and maps. Among them, the Lusternik–Schnirelmann (LS) category, introduced in the 1930s, and the sectional category of a map, developed later by Švarc \cite{Sva}, Berstein and Ganea \cite{B-G}, stand out as central tools with applications ranging from critical point theory to robotics. A significant leap came with Farber’s introduction of topological complexity \cite{F}, which reframed the sectional category in terms of the free path fibration to quantify the complexity of motion planning algorithms.

The presence of symmetries naturally led to equivariant versions of these invariants. Fadell \cite{Fadelleqcat} defined the equivariant LS category, later studied by Marzantowicz \cite{Eqlscategory}, while Colman and Grant \cite{EqTC} introduced equivariant topological complexity as the equivariant sectional category of equivariant fibrations. These developments extended the reach of the LS category and topological complexity to the context of $G$-spaces under the action of compact Lie groups.

The homotopic distance between continuous maps, introduced by
Maciás--Virgos and Mosquera--Lois~\cite{macias2022homotopic}, provides a
unifying framework for several classical invariants in homotopy theory,
including the Lusternik--Schnirelmann category and topological
complexity. 
In this paper we propose the \emph{equivariant homotopic distance}
$D_G(f,g)$ for $G$--maps $f,g\colon X\to Y$. The definition is a natural
extension of the nonequivariant case, and Theorem~\ref{thm: DG equals
secat} shows that $D_G(f,g)$ can be expressed as a particular instance
of the equivariant sectional category $\sct_G$. This places $D_G$
within a well--established framework, and many basic properties follow
as natural analogues of those in the nonequivariant setting. At the same
time, phrasing the theory in terms of $D_G$ provides a convenient
language centered on pairs of maps, in which several structural features
of the equivariant situation become transparent.

\medskip

The main contributions of this paper include a categorical proof of the \emph{triangle inequality} in the equivariant setting, which avoids delicate covering arguments. We establish cohomological and dimension--connectivity bounds, including Borel--cohomology estimates that can be strictly sharper than the corresponding orbit space bounds. Further, we apply these results to \emph{Hopf $G$--spaces}, obtaining the equality $\TC_G(X) = \ct_G(X)$ under natural hypotheses, as an equivariant analogue of the theorem of Lupton--Scherer in the nonequivariant case. Finally, we extend our results to \emph{equivariant fibrations}, thereby generalizing inequalities of Varadarajan \cite{V} and Farber-Grant \cite{F-G}.

\medskip

Taken together, these results indicate that the equivariant homotopic
distance $D_G$ not only extends the nonequivariant theory but also
offers a flexible and effective language for analyzing equivariant
Lusternik--Schnirelmann category and topological complexity.

The paper is organized as follows. In \Cref{sec:eq-secat} we recall the basic notions of equivariant sectional category. \Cref{sec: eq homotopic distance} introduces the equivariant homotopic distance and establishes its relationship with equivariant sectional category, together with cohomological and dimensional bounds. \Cref{sec: properties} is devoted to fundamental properties, while \Cref{sec: triangle inequality} proves the triangle inequality by means of a categorical approach. In \Cref{sec: Hopg G-space} we apply the theory to Hopf $G$-spaces, obtaining an equivariant version of Lupton–Scherer’s theorem and estimates for Hopf $G$-spheres. Finally, \Cref{sec: equivariant fibrations} deals with equivariant fibrations, relating the homotopic distance of fibre-preserving maps to the invariants of the fibres and the base.

\section{Preliminaries: equivariant sectional category}\label{sec:eq-secat}

Throughout this paper, we work in the category of $G$-spaces, where $G$ denotes a compact Lie group. We use the symbol $\simeq_G$ to indicate $G$-homotopy and, when appropriate, a $G$-homotopy equivalence. 

The sectional category of a Hurewicz fibration was originally introduced by \v{S}varc in  \cite{Sva}. Later, Colman and Grant \cite{EqTC} extended this concept to the equivariant setting, defining the equivariant sectional category. 

\begin{definition}[{\cite[Definition 4.1]{EqTC}}]\label{def:eqsecat}
The equivariant sectional category of a $G$-map $p:E\to B$, denoted $\sct_G(p)$, is the least nonnegative integer $n$ such that there is a cover of $B$ by $n+1$ invariant open subsets $\{U_0, \ldots, U_n\}$ and, for each $i$, a local $G$-homotopy section $s_i \colon U_i \to E$ of $p$, that is, the $G$-map $s_i$ is such that the following diagram commutes up to $G$-homotopy
$$\xymatrix{
{U_i} \ar@{^(->}[rr]^{\mathrm{inc}_{U_i}} \ar[dr]_{s_i} & & {B} \\
 & {E} \ar[ur]_p &  }$$
\noindent where $\mathrm{inc}_{U_i} : U_i \hookrightarrow B$ is the inclusion map. If no such integer exists, we set $\sct_G(p):=\infty.$ 
\end{definition}

\begin{remark}\
\begin{enumerate}

\item Clearly, if $p\simeq _G p'$, then $\sct_G(p)=\sct_G(p').$

\item As observed by Colman and Grant, if $p$ is a $G$-fibration, then the triangles in the definition above can be chosen to commute strictly.

\item When $G$ is the trivial group, $\sct_G(p)$ reduces to the classical (nonequivariant) sectional category of $p$, denoted by $\sct(p)$.
\end{enumerate}
\end{remark}

Let us now recall some fundamental properties of the $G$-sectional category, which will be used throughout the paper.

\begin{lemma}\label{ineq}
Let 
$$\xymatrix{
{E} \ar[rr]^{\alpha } \ar[dr]_p & & {E'} \ar[dl]^{p'} \\
 & {B} &
}$$
be a diagram of $G$-spaces and $G$-maps that commutes up to $G$-homotopy. Then
\[
\sct_G(p') \leq \sct_G(p).
\]
\end{lemma}

\begin{proof}
Let $U \subseteq B$ be an invariant open subset, and let $\sigma \colon U \to E$ be a $G$-map such that 
\(p \circ \sigma \simeq_G \mathrm{inc}_U\).  
Then the composite $\sigma' := \alpha \circ \sigma \colon U \to E'$ satisfies
\[
p' \circ \sigma' \;=\; p' \circ \alpha \circ \sigma \;\simeq_G\; p \circ \sigma \;\simeq_G\; \mathrm{inc}_U.
\]
Hence, every local $G$-homotopy section of $p$ induces one for $p'$. Applying this to all subsets in a cover of $B$ realizing $\sct_G(p)$ gives the inequality  $\sct_G(p') \leq \sct_G(p)$.
\end{proof}

As a further property, we note that the $G$-sectional category is invariant under $G$-homotopy, as stated below:

\begin{proposition}\label{G-equiv-Gsecat}
Let 
$$\xymatrix{
{E} \ar[rr]^{\alpha }_{\simeq _G} \ar[d]_p & & {E'} \ar[d]^{p'} \\
{B} \ar[rr]_{\beta }^{\simeq _G} & & {B',}
}$$
be a $G$-homotopy commutative diagram, where $\alpha$ and $\beta$ are $G$-homotopy equivalences. Then 
\[
\sct_G(p)=\sct_G(p').
\]
\end{proposition}

\begin{proof}
First suppose that $B=B'$ and $\beta=1_B$. By Lemma~\ref{ineq} we have $\sct_G(p') \leq \sct_G(p)$.  
Applying the same argument to the $G$-homotopy inverse of $\alpha$ gives the reverse inequality, and hence $\sct_G(p)=\sct_G(p')$.  

Next suppose that $E=E'$ and $\alpha=1_E$. Let $U\subseteq B$ be an invariant open subset and let $s:U\to E$ be a local $G$-homotopy section of $p'$. Define $V:=\beta^{-1}(U)$ and 
\[
s' := s \circ (\beta|_V) \colon V \to E.
\]
Then
\[
\beta \circ p \circ s' \;\simeq_G\; \mathrm{inc}_U \circ (\beta|_V) \;=\; \beta \circ \mathrm{inc}_V,
\]
so $p\circ s' \simeq_G \mathrm{inc}_V$. This shows $\sct_G(p) \leq \sct_G(p')$.  
Applying the same reasoning with the $G$-homotopy inverse of $\beta$ yields the equality $\sct_G(p)=\sct_G(p')$.  

For the general case, combining the two particular cases gives
\[
\sct_G(p)=\sct_G(\beta\circ p)=\sct_G(p'\circ \alpha)=\sct_G(p').
\]
\end{proof}

Another useful feature, which will play a role in later arguments, is the subadditivity of the equivariant sectional category.  Its proof can be found in \cite[Prop. 2.11]{SD}.

\begin{proposition}\label{sub}
Let $p_i:E_i\to B_i$ be a $G$-fibration, for $i\in \{1,2\}$.  
Then the product map
\[
p_1\times p_2:E_1\times E_2 \longrightarrow B_1\times B_2
\]
is a $G$-fibration, where $G$ acts diagonally on both $E_1\times E_2$ and $B_1\times B_2$.  
Moreover, if $B_1\times B_2$ is completely normal and $B_1,B_2$ are Hausdorff, then
\[
\sct _G(p_1\times p_2)\leq \sct _G(p_1)+\sct _G(p_2).
\]
\end{proposition}

\bigskip
Fadell \cite{Fadelleqcat} introduced the notion of the $G$-equivariant Lusternik–Schnirelmann category for $G$-spaces. 
This concept was further developed by Marzantowicz \cite{Eqlscategory}, Clapp and Puppe \cite{Clapp-Puppe}, Colman \cite{Colmaneqcat}, and Angel, Colman, Grant, and Oprea \cite{Angel-Colman-Grant-Oprea-Moritainv-eqcat}. 
This homotopy invariant of a $G$-space $X$ is denoted by $\ct_G(X)$.  

Before stating the definition, we recall the notion of a $G$-categorical set. 
A $G$-invariant open subset $U \subseteq X$ is called \emph{$G$-categorical} if the inclusion $i_{U}\colon U \hookrightarrow X$ is $G$-homotopic to a $G$-map whose image lies entirely within a single orbit.  

\begin{definition}
For a $G$-space $X$, the $G$-equivariant category $\ct_G(X)$ is the least nonnegative integer $n$ such that $X$ can be covered by $n+1$ $G$-categorical sets. 
If no such integer exists, we set $\ct_G(X) := \infty$.
\end{definition}

Although standard, we briefly recall some basic facts about fixed points under closed subgroups, 
both to fix notation and for later use. 
Let $H$ be a closed subgroup of $G$, and let $X$ be a $G$-space. 
The set of $H$-fixed points of $X$ is denoted by $X^H$ and defined as
\[
X^H := \{x \in X \mid h \cdot x = x \text{ for all } h \in H\}.
\]
If $f \colon X \to Y$ is a $G$-map, then $f(X^H) \subseteq Y^H$, so the restriction of $f$ to $X^H$ induces a map 
\[
f^H \colon X^H \to Y^H.
\]

Let $\mathbf{GTop}$ denote the category of $G$-spaces and $G$-maps, and $\mathbf{Top}$ the category of topological spaces and continuous maps. 
The construction $X \mapsto X^H$ thus defines a functor
\[
(-)^H \colon \mathbf{GTop} \to \mathbf{Top}.
\]
It is well known that this functor is naturally isomorphic to the hom-functor 
$$\mathrm{Hom}_{\mathbf{GTop}}(G/H, -) \colon \mathbf{GTop} \to \mathbf{Top},$$ 
and therefore preserves all small limits. 
In particular, $(-)^H$ preserves pullbacks and small products.

The fixed-point construction also allows us to formulate a natural equivariant analogue of connectedness:

\begin{definition}
A $G$-space $X$ is said to be \emph{$G$-connected} if, for every closed subgroup $H$ of $G$, the fixed-point set $X^H$ is path-connected.
\end{definition}

This notion provides a useful bridge between equivariant sectional category and equivariant LS category:

\begin{remark}\label{fund}
If $x_0\in X$ is a fixed point of the $G$-action and $X$ is $G$-connected, then the equivariant sectional category of the inclusion 
\(i\colon \{x_0\} \to X\) satisfies
\[
\sct_G(i\colon \{x_0\} \to X) = \ct_G(X);
\]
see \cite[Corollary 4.7]{EqTC}.  
Moreover, when $G$ is the trivial group, this recovers the classical invariant, i.e.\ \(\ct_G(X)=\ct(X)\).
\end{remark}

\begin{corollary}
Let $p \colon E \to B$ be a $G$-map such that $B$ is $G$-connected and $E^G \neq \varnothing$. Then
\[
\sct_G(p) \;\leq\; \ct_G(B).
\]
\end{corollary}

\begin{proof}
Choose $e_0 \in E^G$ and set $b_0 := p(e_0) \in B^G$. 
Then the following diagram of $G$-maps commutes:
\[
\xymatrix{
{\{e_0\}} \ar@{^(->}[rr] \ar[dr]_{b_0} & & {E} \ar[dl]^p \\
 & {B} & }
\]
The result follows directly from Lemma~\ref{ineq} together with Remark~\ref{fund}.
\end{proof}

We conclude this section by recalling two standard bounds for the equivariant sectional category.  
First, we record a slightly modified version of the equivariant homotopy dimension–connectivity upper bound, as stated in \cite{G-symmetrized,SD}. 
Since the proof is essentially the same as in the original references, we omit it.  
Throughout, a \emph{Serre $G$-fibration} will mean a $G$-map $p \colon E \to B$ satisfying the $G$-homotopy lifting property with respect to all $G$-CW complexes. 
Clearly, every Serre $G$-fibration is in particular a $G$-fibration.

\begin{proposition}\label{connect}
Let $p \colon E \to B$ be a Serre $G$-fibration, and suppose that $B$ is a $G$-CW complex with $\dim(B) \geq 2$. 
Assume that for every closed subgroup $H \subseteq G$, the fixed-point map $p^H \colon E^H \to B^H$ is an $m$-equivalence. 
Then
\[
\sct_G(p) \;<\; \frac{\mathrm{hdim}(B) + 1}{m + 1}.
\]
Here $\mathrm{hdim}(B)$ denotes the minimal dimension $\dim(B')$ such that $B'$ is a $G$-CW complex $G$-homotopy equivalent to $B$.
\end{proposition}

Finally, we state the classical cohomological lower bound for the equivariant sectional category, 
originally established by Arora and Daundkar in \cite{SD}. 
We begin by fixing notation for Borel cohomology. 
Let $EG \to BG$ be the universal principal $G$-bundle. 
For a $G$-space $X$, its homotopy orbit space is
\[
X^h_G := EG \times_G X,
\]
and its Borel cohomology with coefficients in a commutative ring $R$ is
\[
H^*_G(X;R) := H^*(X^h_G;R).
\]
If $f \colon X \to Y$ is a $G$-map, we denote by $f^h_G \colon X^h_G \to Y^h_G$ the induced map between the corresponding homotopy orbit spaces.

\begin{proposition}\label{Gcohom}
Let $p \colon E \to B$ be a $G$-map. 
Suppose there exist cohomology classes $u_1,\dots,u_k \in H^*_G(B;R)$ such that
\[
(p^h_G)^*(u_1) = \cdots = (p^h_G)^*(u_k) = 0,
\quad\text{and}\quad
u_1 \smile \cdots \smile u_k \neq 0.
\]
Then $\sct_G(p) > k$.
\end{proposition}

\noindent A proof of this result can be found in \cite[Theorem~2.2]{SD}.

Together with the dimension–connectivity upper bound, this cohomological lower bound will serve as a key tool in the sections that follow.

\section{Equivariant homotopic distance}\label{sec: eq homotopic distance}

In this section we introduce the notion of \emph{equivariant homotopic distance}, which plays a central role in our work. 
This concept extends the homotopic distance introduced by Macías-Virgós and Mosquera-Lois in \cite{macias2022homotopic}, adapting it to the setting of spaces endowed with a continuous action of a topological group.

\begin{definition}
Let $X$ and $Y$ be $G$-spaces, and let $f,g \colon X \to Y$ be continuous $G$-maps.  
The \emph{equivariant homotopic distance} between $f$ and $g$, denoted $\D_G(f,g)$, is defined as the least nonnegative integer $n$ such that there exists an open cover $\{U_0, \dots, U_n\}$ of $X$ with each $U_i$ $G$-invariant, and such that 
\[
f|_{U_i} \simeq_G g|_{U_i}, \quad 0 \leq i \leq n.
\]
If no such integer exists, we set $\D_G(f,g) = \infty$.
\end{definition}

In this sense, the equivariant homotopic distance provides a quantitative measure of how far two $G$-maps are from being $G$-homotopic. 
Some of its basic properties follow immediately from the definition.

\begin{remark}
Let $f,g,f',g' \colon X \to Y$ be $G$-maps. Then:
\begin{enumerate}
\item If $G$ acts trivially on $X$ and $Y$, then $\D_G(f,g) = D(f,g)$, where $D(f,g)$ is the classical homotopic distance.  
In general, one always has $\D(f,g) \leq \D_G(f,g)$.

\item $\D_G(f,g) = \D_G(g,f)$.

\item $\D_G(f,g) = 0$ if and only if $f \simeq_G g$.

\item If $f \simeq_G f'$ and $g \simeq_G g'$, then $\D_G(f,g) = \D_G(f',g')$.
\end{enumerate}
\end{remark}

\bigskip
To relate the equivariant homotopic distance to the equivariant sectional category, we consider the pullback of the free path fibration along the map 
$(f,g) \colon X \to Y \times Y$. This leads to the following characterization.

\begin{theorem}\label{thm: DG equals secat}
Let $f,g:X\to Y$ be $G$-maps with the following pullback diagram:
\begin{equation}\label{eq:pullback of free path fib}
\xymatrix{
{\mathcal{P}(f,g)} \ar[d]_{\bar{\pi}_Y} \ar[r]^{p} & {Y^I} \ar[d]^{\pi_Y} \\
{X} \ar[r]_{(f ,g )} & {Y\times Y.} }    
\end{equation}
Then
  $\D_G(f,g)=\sct_G(\bar{\pi}_Y)$.
\end{theorem}

\begin{proof}
Recall that 
\[
\mathcal{P}(f,g)=\{(x,\gamma)\in X\times Y^I \mid f(x)=\gamma(0), \ g(x)=\gamma(1)\}.
\]
Suppose $U\subseteq X$ is a $G$-invariant open subset. 
If $\bar{\pi}_Y$ admits a local $G$-section $s:U\to \mathcal{P}(f,g)$, then $s(u)=(u,\gamma_u)$ for some path $\gamma_u$, and the homotopy $H(u,t)=\gamma_u(t)$ shows that $f|_U\simeq_G g|_U$.  
This gives $\D_G(f,g)\leq \sct_G(\bar{\pi}_Y)$.

Conversely, if $f|_U\simeq_G g|_U$ via a $G$-homotopy $H:U\times I\to Y$, then the map $h:U\to Y^I$, $h(u)(t)=H(u,t)$, satisfies $\pi_Y\circ h=(f,g)|_U$. By the universal property of pullbacks, $h$ lifts to a local $G$-section $s:U\to \mathcal{P}(f,g)$ of $\bar{\pi}_Y$.  
Thus $\sct_G(\bar{\pi}_Y)\leq \D_G(f,g)$.
Putting the two inequalities together yields the result. The argument is entirely analogous to the classical case (see \cite{macias2022homotopic}), but we have included the essential steps for completeness.
\end{proof}

As a direct application, the equivariant topological complexity of a $G$-space can be expressed as the equivariant homotopic distance between the two canonical projections.  

\begin{corollary}\label{cor: DG equals TCG}
Let $X$ be a $G$-space, and let $pr_1^X, pr_2^X \colon X \times X \to X$ denote the coordinate projections. Then 
\[
\TC_G(X) \;=\; \D_G(pr_1^X, pr_2^X).
\]
\end{corollary}

\begin{proof}
Observe that the map $(pr_1^X, pr_2^X) \colon X \times X \to X \times X$ is the identity.  
Consequently, in the pullback diagram \eqref{eq:pullback of free path fib} we have $\bar{\pi}_X = \pi_X$.  
Applying Theorem~\ref{thm: DG equals secat} therefore yields
\[
\TC_G(X) \;=\; \sct_G(\pi_X) \;=\; \sct_G(\bar{\pi}_X) \;=\; \D_G(pr_1^X, pr_2^X).
\qedhere
\]
\end{proof}

\begin{remark}
This result is the direct equivariant analogue of the classical identity 
$\TC(X) = D(pr_1^X,pr_2^X)$, 
which expresses topological complexity as a homotopic distance.  
Thus, $\D_G$ provides a natural extension of this relationship to the equivariant setting, 
confirming that $\TC_G$ can be fully understood as a special case of equivariant homotopic distance.
\end{remark}

Next, we examine a natural connection between the equivariant homotopic distance and the equivariant LS category. 
This result should be seen as the equivariant analogue of the classical identity 
$\ct(X) = \D(i_1,i_2)$, though certain subtleties arise due to the presence of fixed points and the notion of $G$-connectedness.  

Suppose $X$ is a $G$-space with a global fixed point $x_0 \in X^G$. 
The product $X \times X$ becomes a $G$-space under the diagonal $G$-action. 
Let $c_{x_0} \colon X \to X$ denote the constant $G$-map, and define the $G$-maps
\[
i_1(x) = (x,x_0) = (id_X,c_{x_0})(x),
\qquad
i_2(x) = (x_0,x) = (c_{x_0},id_X)(x).
\]

\begin{proposition}
Let $X$ be a $G$-connected $G$-space with $X^G \neq \varnothing$. 
Then
\[
\ct_G(X) \;=\; \D_G(i_1,i_2).
\]
\end{proposition}

\begin{proof}
Suppose $U \subseteq X$ is a $G$-categorical subset. 
Since $X$ is $G$-connected and admits a global fixed point $x_0 \in X^G$, 
\cite[Lemma~3.14]{EqTC} provides a $G$-homotopy $H \colon U \times I \to X$ with 
$H(u,0)=u$ and $H(u,1)=x_0$ for all $u\in U$.  
Define
\[
F \colon U \times I \to X \times X, 
\qquad 
F(u,t) =
\begin{cases}
(H(u,2t), x_0), & 0 \leq t \leq \tfrac{1}{2}, \\[4pt]
(x_0, H(u, 2-2t)), & \tfrac{1}{2} \leq t \leq 1,
\end{cases}
\]
which is a $G$-map yielding a $G$-homotopy between $i_1|_U$ and $i_2|_U$.  
Applying this to a $G$-categorical cover of $X$ gives 
$\D_G(i_1,i_2) \leq \ct_G(X)$.

Conversely, let $U \subseteq X$ be a $G$-invariant open set admitting a $G$-homotopy 
$H' \colon U \times I \to X \times X$ 
between $i_1|_U$ and $i_2|_U$.  
Projecting onto the first factor, define
\[
F'(u,t) = pr_1^X\big(H'(u,t)\big).
\]
Then $F'$ is a $G$-homotopy between $\mathrm{inc}_U$ and the constant map $c_{x_0}$, showing that $U$ is $G$-categorical.  
Applying this to a $G$-invariant open cover yields $\D_G(i_1,i_2) \geq \ct_G(X)$.

Hence $\D_G(i_1,i_2) = \ct_G(X)$.
\end{proof}

\begin{remark}
This identity mirrors the classical relationship between LS category and homotopic distance, 
with the additional requirement of $G$-connectedness and the presence of a global fixed point ensuring that the equivariant structure behaves as expected. 
\end{remark}

\bigskip

We now turn to the cohomological aspects of the invariant. We establish several lower bounds for the equivariant homotopic distance, formulated in terms of Borel cohomology with coefficients in a fixed commutative ring $R$.  These results extend the classical cohomological estimates for homotopic distance to the equivariant setting.

\begin{theorem}\label{thm: cohomological lower bounds}
Let $f, g \colon X \to Y$ be $G$-maps, and let $\Delta \colon Y \to Y \times Y$ denote the diagonal map. Then:
\begin{enumerate}
\item Suppose there exist classes $z_1, \dots, z_k \in H_G^*(Y \times Y; R)$ such that 
\((\Delta^h_G)^*(z_i) = 0\) for all $1 \leq i \leq k$, and 
\((f, g)^*(z_1 \smile \dots \smile z_k) \neq 0\).  
Then
\[
\D_G(f, g) \;\geq\; k.
\]

\item Let 
\[
\mathcal{J}_{f,g} := \mathrm{Im}\!\big((f^h_G)^* - (g^h_G)^*\big) \;\subseteq\; H_G^*(X;R).
\]
Then
\[
\D_G(f, g) \;\geq\; \mathrm{nil}(\mathcal{J}_{f,g}),
\]
where $\mathrm{nil}(\mathcal{J}_{f,g})$ denotes the nilpotency of the ideal $\mathcal{J}_{f,g}$.
\end{enumerate}
\end{theorem}

\begin{proof}
(1) By \Cref{thm: DG equals secat}, we have $\D_G(f,g)=\sct_G(\bar{\pi}_Y)$.  
Applying Proposition~\ref{Gcohom} (see also \cite[Theorem~2.2]{SD}), it follows that
\[
\D_G(f,g) \;\geq\; \mathrm{nil}\!\big(\ker((\bar{\pi}_Y)^h_G)^*\big).
\]
Now assume $z_1, \dots, z_k \in H_G^*(Y \times Y; R)$ satisfy $(\Delta^h_G)^*(z_i) = 0$ for all $i$, and $(f, g)^*(z_1 \smile \dots \smile z_k) \neq 0$.  
Set $y_i := ((f,g)^h_G)^*(z_i) \in H_G^*(X;R)$.  
Using the pullback diagram \eqref{eq:pullback of free path fib} and the relation $\Delta^h_G \simeq_G (\pi_Y)^h_G$, we compute:
\[
((\bar{\pi}_Y)^h_G)^*(y_i) 
= (p^h_G)^*((\Delta^h_G)^*(z_i)) 
= 0.
\]
Thus each $y_i$ lies in $\ker((\bar{\pi}_Y)^h_G)^*$.  
Moreover,
\[
y_1 \smile \dots \smile y_k
= ((f,g)^h_G)^*(z_1 \smile \dots \smile z_k) \;\neq\; 0.
\]
Hence $\mathrm{nil}(\ker((\bar{\pi}_Y)^h_G)^*) \geq k$, giving the desired inequality.

\smallskip
(2) Suppose $\D_G(f,g) \leq n$.  
Then there exists a $G$-invariant open cover $\{U_0, \dots, U_n\}$ of $X$ such that $f|_{U_i} \simeq_G g|_{U_i}$ for each $i$.  
For every $0 \leq i \leq n$, consider the long exact sequence in Borel cohomology of the pair $(X,U_i)$, fitting into the commutative diagram
\[
\xymatrix{
& & & H^{m_i}_G(Y) \ar[d]_{(f^h_G)^*-(g^h_G)^*} \ar[drr]^{\hspace{0.7cm}((f|_{U_i})^h_G)^*-((g|_{U_i})^h_G)^*} & & \\
\cdots \ar[r] & H^{m_i}_G(X,U_i) \ar[rr]_{((q_i)^h_G)^*} & & H^{m_i}_G(X) \ar[rr]_{((\mathrm{inc}_{U_i})^h_G)^*} & & H^{m_i}_G(U_i) \ar[r] & \cdots
}
\]
It follows that $\mathcal{J}_{f,g} \subseteq \ker((\mathrm{inc}_{U_i})^*)$.  
Therefore, for each $y_i \in \mathcal{J}_{f,g}$, there exists $x_i \in H^{m_i}(X,U_i)$ with $((q_i)^h_G)^*(x_i) = y_i$.  
Now compute:
\[
\begin{aligned}
y_0 \smile \dots \smile y_n
&= ((q_0)^h_G)^*(x_0) \smile \dots \smile ((q_n)^h_G)^*(x_n) \\
&= (q^h_G)^*(x_0 \smile \dots \smile x_n) \\
&= (q^h_G)^*(0) \;=\; 0,
\end{aligned}
\]
since $x_0 \smile \dots \smile x_n \in H^m(X,X) = 0$ (with $m = \sum m_i$).  
Thus $\mathrm{nil}(\mathcal{J}_{f,g}) \leq n$, and the result follows.
\end{proof}

%\begin{remark}
%The proof of part $(1)$ of \Cref{thm: cohomological lower bounds} relies on the cohomological lower bound established for the relative sectional category by Garc\'ia-Calcines in \cite{GCrelsecat}.    
%\end{remark}

\begin{example}${}$

\begin{enumerate}
    \item In the case of the maps $i_1,i_2 \colon X \to X \times X$ introduced earlier, we have 
    \(\mathcal{J}_{i_1,i_2} = H_G^*(X)\).  
    Hence, the lower bound in part~(2) of \Cref{thm: cohomological lower bounds} recovers the cohomological lower bound for the equivariant LS category established by the first author together with Arora in \cite{SD}.
    
    \item If $f = pr_1^X$ and $g = pr_2^X$, then $(pr_1^X, pr_2^X) = id_{X \times X}$.  
    In this case, part~(1) of \Cref{thm: cohomological lower bounds} recovers the cohomological lower bound for the equivariant topological complexity proved by Colman and Grant in \cite[Theorem~5.15]{EqTC}.  
    Furthermore, in part~(2), if $R$ is a field, then 
    \[
    \mathcal{J}_{pr_1^X,pr_2^X} = \ker\!\big((\Delta_G^h)^*\big),
    \] 
    and we again recover the same lower bound from \cite[Theorem~5.15]{EqTC}.
\end{enumerate}
\end{example}

Although equivariant cohomology, such as Borel cohomology, provides a natural framework for studying group actions, 
the complexity of computing cup products often makes it difficult to apply in practice. 
To overcome this difficulty, one can instead use ordinary (nonequivariant) cohomology to obtain effective lower bounds 
for the equivariant homotopic distance.

\begin{proposition}\label{prop: cohomological bound 2}
Let $\bar{\pi}_Y$ be as in the pullback diagram \eqref{eq:pullback of free path fib}. 
If $\bar{\pi}_Y' \colon E/G \to X/G$ denotes the induced map on orbit spaces, then
\[
\D_G(f,g)\;\geq\; \mathrm{nil}\!\big(\ker((\bar{\pi}_Y')^*)\big),
\]
where $(\bar{\pi}_Y')^* \colon H^*(X/G) \to H^*(E/G)$ is the induced map in singular cohomology.
\end{proposition}

\begin{proof}
By \cite[Theorem~2.5(1)]{SD}, the map $\bar{\pi}_Y' \colon E/G \to X/G$ is a fibration and 
\(\sct(\bar{\pi}_Y') \leq \sct_G(\bar{\pi}_Y)\). 
Since $\D_G(f,g) = \sct_G(\bar{\pi}_Y)$ by \Cref{thm: DG equals secat}, the inequality follows directly from \cite[Theorem~4]{Sva}.
\end{proof}

Generalizing Schwarz’s dimension–connectivity upper bound for sectional category, Grant established in \cite[Theorem 3.5]{G-symmetrized} the corresponding equivariant analogue for the equivariant sectional category. This result was later extended by Daundkar and Arora \cite{SD}, who derived an equivariant homotopy dimension–connectivity upper bound.

%We briefly recall the notion of equivariant homotopy dimension.  
%If $X$ is a $G$-CW complex, the \emph{$G$-homotopy dimension} of $X$, denoted $\mathrm{hdim}_G(X)$, is the minimal dimension of a $G$-CW complex that is $G$-homotopy equivalent to $X$.

Using \cite[Theorem 2.12]{SD}, we now establish the equivariant homotopy dimension–connectivity upper bound for the equivariant homotopic distance.

\begin{proposition}\label{prop: dim-conn upper bound}
Let $f,g \colon X \to Y$ be $G$-maps, with $X$ a $G$-CW complex of dimension at least $2$.  
Suppose that $Y$ is $s$–$G$-connected; that is, for every closed subgroup $H \leq G$, the fixed-point set $Y^H$ is $s$-connected. Then
\[
\D_G(f,g) \;<\; \frac{\mathrm{hdim}_G(X)+1}{s+1}.
\]
\end{proposition}

\begin{proof}
By \Cref{thm: DG equals secat}, we have $\D_G(f,g)=\sct_G(\bar{\pi}_Y)$.  
The inequality then follows from \Cref{connect}. Indeed, for every closed subgroup $H \leq G$, the map $\pi_Y^H$ is an $s$-equivalence, and since the fixed-point functor $(-)^H$ preserves pullbacks, it follows that $\bar{\pi}_Y^H$ is also an $s$-equivalence.    
\end{proof}

\section{Properties}\label{sec: properties}
In the previous section, we established some foundational results on the equivariant homotopic distance, including cohomological lower bounds and dimension–connectivity type upper bounds. We now turn to a more systematic study of its structural properties. Many of the statements presented here resemble their nonequivariant counterparts, but their proofs require careful attention to the presence of the group action. In some cases, the arguments adapt smoothly, while in others new subtleties appear—such as the behavior of invariant open covers, the role of fixed point sets, or the interplay with equivariant lifting properties. For completeness, and to make these subtleties explicit, we include full proofs rather than leaving them as straightforward adaptations

Our analysis begins with composition inequalities for $G$-maps, and then develops a series of connections with the equivariant Lusternik–Schnirelmann category and the equivariant topological complexity. Along the way, we obtain comparison results, product-type estimates, subgroup restrictions, and finally establish equivariant homotopy invariance.

The proof of the following proposition is completely analogous to the nonequivariant case, which was proved in section $3$ of \cite{macias2022homotopic}.
\begin{proposition}\label{prop:composition-inequalities}
Let $f,g\colon X \to Y$ be $G$-maps.
\begin{enumerate}
    \item If $h\colon Y \to Z$ is a $G$-map, then 
$\D_G(h\circ f,\,h\circ g)\;\leq\;\D_G(f,g).$

    \item If $k\colon Z \to X$ is a $G$-map, then 
$\D_G(f\circ k,\,g\circ k)\;\leq\;\D_G(f,g).$
\end{enumerate}
\end{proposition}

\medskip

At this point, we introduce the equivariant version of the Lusternik–Schnirelmann category for $G$-maps.

\begin{definition}
Let $f\colon X \to Y$ be a $G$-map.  
The \emph{equivariant LS category} of $f$, denoted $\ct_G(f)$, is the least integer $n\geq 0$ such that $X$ admits a cover $\{U_i\}_{i=0}^n$ by $n+1$ $G$-invariant open subsets with the property that, for each $i$, the restriction $f|_{U_i}$ is $G$-homotopic to a $G$-map $U_i\to Y$ whose image lies in the orbit of some $y_i\in Y$.  
If no such $n$ exists, we set $\ct_G(f)=\infty$.
\end{definition}

In particular, for any $G$-space $X$, one has
\[
\ct_G(X)=\ct_G(id_X).
\]

This invariant can also be described in terms of the equivariant homotopic distance to a constant map:

\begin{proposition}\label{ct-mor-car}
Let $f\colon X\to Y$ be a $G$-map, where $Y$ is $G$-connected and $Y^G\neq \emptyset$.  
If $y_0\in Y^G$, then
\[
\ct_G(f)=\D_G(f,c_{y_0}),
\]
where $c_{y_0}\colon X\to Y$ denotes the constant map at $y_0$.
\end{proposition}

\begin{proof}
The inequality $\ct_G(f)\leq \D_G(f,c_{y_0})$ is immediate.  
For the converse, suppose $U\subseteq X$ is a $G$-invariant open subset such that $f|_U\simeq_G c$, where $c\colon U\to Y$ takes values in the orbit $O(y)$ of some $y\in Y$.  
By \cite[Lemma~3.14]{EqTC}, the map $c$ is $G$-homotopic to $c_{y_0}|_U\colon U\to Y$.  
Thus $f|_U\simeq_G c_{y_0}|_U$, which establishes the reverse inequality.
\end{proof}

\begin{remark}\label{rem-ct-mor-car}${}$

\begin{enumerate}
    \item If $y_0\in Y^G$, then the inequality 
    \[
        \ct_G(f)\leq \D_G(f,c_{y_0})
    \]
    holds even when $Y$ is not $G$-connected.  
    More generally, $\ct_G(f)\leq \D_G(f,c)$ whenever $c\colon X\to Y$ is a $G$-map with image contained in an orbit $O(y)$, without any assumptions on $Y$.
    \item Under the hypotheses of Proposition~\ref{ct-mor-car}, one actually has 
    \[
        \ct_G(f)=\D_G(f,c),
    \]
    for any $G$-map $c\colon X\to Y$ with image contained in an orbit $O(y)$ (for some, equivalently any, $y\in Y$).
\end{enumerate}
\end{remark}

\bigskip
We proceed to establish inequalities relating the equivariant LS category to the equivariant homotopic distance.  
The first result provides upper bounds for $\ct_G(f)$ in terms of the equivariant categories of the source and target spaces.  
The second one yields a product-type estimate for the equivariant homotopic distance between two $G$-maps in terms of their equivariant LS categories.

\begin{corollary}
Let $f\colon X \to Y$ be a $G$-map. Then:
\begin{enumerate}
    \item If $Y$ is $G$-connected and $Y^G\neq \emptyset$, then $\ct_G(f)\;\leq\;\ct_G(Y).$
    
    \item If $X$ is $G$-connected and $X^G\neq \emptyset$, then $\ct_G(f)\;\leq\;\ct_G(X).$
\end{enumerate}
\end{corollary}

\begin{proof}
(1) Let $y_0\in Y^G$. Then
\[
\ct_G(f)=\D_G(f,c_{y_0})
       =\D_G(id_Y\circ f,\,c_{y_0}\circ f)
       \leq \D_G(id_Y,c_{y_0})
       =\ct_G(Y).
\]

(2) Let $x_0\in X^G$. Then $f(x_0)\in Y^G$, and hence
\[
\ct_G(f)\leq \D_G(f,c_{f(x_0)})
          =   \D_G(f\circ id_X,\,f\circ c_{x_0})
          \leq \D_G(id_X,c_{x_0})
          =   \ct_G(X).
\]
\end{proof}

\begin{corollary}\label{ineq-D-cat}
Let $f,g\colon X \to Y$ be $G$-maps.  
If $Y$ is $G$-connected and $Y^G \neq \emptyset$, then 
\[
    \D_G(f,g)\;\leq\;(\ct_G(f)+1)\cdot(\ct_G(g)+1).
\]
\end{corollary}

\begin{proof}
By Proposition~\ref{ct-mor-car}, there exist $G$-maps $c,c'\colon X\to Y$ with images contained in orbits $O(y_0)$ and $O(y_0')$ such that
\[
    \ct_G(f)=\D_G(f,c), 
    \qquad 
    \ct_G(g)=\D_G(g,c').
\]
Since $Y$ is $G$-connected and $Y^G\neq \emptyset$, \cite[Lemma~3.14]{EqTC} implies that both $c$ and $c'$ are $G$-homotopic to the same constant $G$-map $c_a\colon X\to Y$ (for some $a\in Y^G$).  
Hence
\[
    \ct_G(f)=\D_G(f,c_a), 
    \qquad 
    \ct_G(g)=\D_G(g,c_a).
\]

Suppose $\ct_G(f)=n$ and $\ct_G(g)=m$.  
Then there exist $G$-invariant open covers 
$\{U_0,\dots,U_n\}$ and $\{V_0,\dots,V_m\}$
of $X$ such that 
$f|_{U_i}\simeq_G c_a|_{U_i}$ and $g|_{V_j}\simeq_G c_a|_{V_j}$.  
For $0\leq i\leq n$ and $0\leq j\leq m$, set $W_{ij}:=U_i\cap V_j$.  
Each $W_{ij}$ is $G$-invariant, and the collection 
$\{W_{ij} \mid 0\leq i\leq n,\; 0\leq j\leq m\}$
forms a $G$-invariant open cover of $X$.  
On each $W_{ij}$ we have 
\[
f|_{W_{ij}}\;\simeq_G\; c_a|_{W_{ij}}\;\simeq_G\; g|_{W_{ij}}.
\]
Thus $\D_G(f,g)\leq (n+1)(m+1)$, as claimed.
\end{proof}

\bigskip
Our next result describes how the equivariant homotopic distance behaves under post-composition with $G$-maps.

\begin{proposition}\label{prop: compo ineq}
Let $h,h'\colon Z \to X$ and $f,g\colon X \to Y$ be $G$-maps.  
If $f\circ h'\simeq_G g\circ h'$, then 
\[
    \D_G(f\circ h,\,g\circ h)\;\leq\;\D_G(h,h').
\]
\end{proposition}

\begin{proof}
Suppose $\D_G(h,h')=n$.  
Then there exists a $G$-invariant open cover 
$\{U_0,\dots,U_n\}$
of $Z$ such that $h|_{U_i}\simeq_G h'|_{U_i}$ for all $i$.  
Since $f\circ h' \simeq_G g\circ h'$, we also have 
\[
    (f\circ h')|_{U_i} \;\simeq_G\; (g\circ h')|_{U_i}, \qquad 0\leq i\leq n.
\]
Therefore, for each $i$,
$(f\circ h)|_{U_i} 
   \;\simeq_G\; (f\circ h')|_{U_i} 
   \;\simeq_G\; (g\circ h')|_{U_i} 
   \;\simeq_G\; (g\circ h)|_{U_i}.$
This shows that $\D_G(f\circ h,\,g\circ h)\leq n$, as required.
\end{proof}

The previous result shows that post-composition cannot increase the equivariant homotopic distance.  
We now turn to comparison inequalities that bound $\D_G(f,g)$ in terms of two classical invariants in the equivariant setting: the topological complexity of the target and the Lusternik–Schnirelmann category of the source.

\begin{proposition}\label{prop: eq hd leq catG tcG}
Let $f,g\colon X \to Y$ be $G$-maps. Then:
\begin{enumerate}
    \item $\D_G(f,g)\;\leq\;\TC_G(Y)$.
    \item If $X$ is $G$-connected and $X^G\neq \emptyset$, then $\D_G(f,g)\;\leq\;\ct_G(X)$.
\end{enumerate}
\end{proposition}

\begin{proof}
(1) By \Cref{thm: DG equals secat}, one has 
$\D_G(f,g)=\sct_G(\bar{\pi}_Y),$
where $\bar{\pi}_Y$ is the pullback of the path fibration $\pi_Y\colon Y^I \to Y\times Y$.  
Hence the inequality $\D_G(f,g)\leq \TC_G(Y)$ follows from \cite[Proposition~4.3]{EqTC}.

(2) Let $h'\colon X\to X$ be a $G$-map with image contained in a single orbit.  
Since $X$ is $G$-connected and $X^G\neq \emptyset$, \cite[Lemma~3.14]{EqTC} ensures that $h'\simeq_G c_a$ for some $a\in X^G$.  
Applying the same lemma again, we deduce
$f\circ h' \;\simeq_G\; c_{f(a)} \;\simeq_G\; c_{g(a)} \;\simeq_G\; g\circ h'.$
Therefore, by Proposition~\ref{prop: compo ineq} with $h=id_X$ and $h'$ as above, we obtain $\D_G(f,g)\leq \ct_G(X)$.
\end{proof}

\begin{remark}
Proposition~\ref{prop: eq hd leq catG tcG} places the equivariant homotopic distance within the same range of classical bounds that relate topological complexity and LS category in the nonequivariant setting.  
In particular, $\D_G(f,g)$ admits both an upper bound coming from the equivariant topological complexity of the target and, under mild hypotheses, another bound given by the equivariant category of the source.  
This highlights its role as a natural intermediary invariant connecting $\ct_G$ and $\TC_G$.
\end{remark}

We now turn how the equivariant homotopic distance behaves under restriction to fixed point sets and subgroup actions.

\begin{proposition}\label{prop:subgroup inequality}
Let $f,g\colon X\to Y$ be $G$-maps, and let $H,K\leq G$ be closed subgroups such that $f^H$ and $g^H$ are $K$-maps. Then
\begin{equation}\label{eq: subgp ineq for hd}
    \D_K(f^H,g^H)\;\leq\;\D_G(f,g).
\end{equation}
In particular,
$\max\bigl\{\D(f^H,g^H),\, \D_K(f,g)\bigr\}\;\leq\;\D_G(f,g).$
\end{proposition}

\begin{proof}
Let $\{U_0,\dots,U_n\}$ be a $G$-invariant open cover of $X$ with $G$-homotopies 
\[
H_i\colon U_i\times I \longrightarrow Y
\]
such that $H_i(u,0)=f(u)$ and $H_i(u,1)=g(u)$.  
Each $U_i$ is also $K$-invariant.  
Define $V_i:=U_i\cap X^H$ for $0\leq i\leq n$.  
Restricting, we obtain $H'_i:=H_i|_{V_i\times I}\colon V_i\times I\to Y^H$, which is $K$-equivariant because each $H_i$ is $G$-equivariant.  
Clearly, $H'_i(u,0)=f^H(u)$ and $H'_i(u,1)=g^H(u)$ for $u\in V_i$.  
This proves \eqref{eq: subgp ineq for hd}.

The special cases follow by setting $K=\{e\}$ (giving $\D(f^H,g^H)\leq \D_G(f,g)$) or $H=\{e\}$ (giving $\D_K(f,g)\leq \D_G(f,g)$).
\end{proof}

From this we recover several results of Colman–Grant \cite{EqTC} in the context of equivariant topological complexity.

\begin{corollary}
Let $X$ be a $G$-space, and let $H,K\leq G$ be closed subgroups such that $X^H$ is $K$-invariant. Then
\begin{equation}\label{eq:TC subgroup inequality}
    \TC_K(X^H)\;\leq\;\TC_G(X).
\end{equation}
In particular,
$\max\bigl\{\TC(X^H),\,\TC_K(X)\bigr\}\;\leq\;\TC_G(X).
$
\end{corollary}

\begin{proof}
For any closed subgroup $H\leq G$, one has $(X\times X)^H = X^H\times X^H$.  
Let $$pr_1^{X^H}, pr_2^{X^H}\colon X^H\times X^H\to X^H$$ denote the coordinate projections.  
Since $X^H$ is $K$-invariant, these maps are $K$-equivariant.  
By Proposition~\ref{prop:subgroup inequality},
\[
\TC_K(X^H) 
    = \D_K(pr_1^{X^H},pr_2^{X^H})
    \leq \D_G(pr_1^X,pr_2^X)
    = \TC_G(X).
\]

Setting $K=\{e\}$ in \eqref{eq:TC subgroup inequality} yields $\TC(X^H)\leq \TC_G(X)$, while setting $H=\{e\}$ gives $\TC_K(X)\leq \TC_G(X)$.
\end{proof}

\begin{corollary}\label{consecu}
Let $X$ be a $G$-connected space. Then:
\begin{enumerate}
    \item If $X^G\neq \emptyset$, then $\TC_G(X)\;\leq\;\ct_G(X\times X).$
    \item If $H=G_x$ for some $x\in X$, then $\ct_H(X)\;\leq\;\TC_G(X).$
\end{enumerate}
\end{corollary}

\begin{proof}
(1) This follows from part~(2) of Proposition~\ref{prop: eq hd leq catG tcG} by taking 
$f=pr_1^X$ and $g=pr_2^X.$

(2) Let $x_0\in X^H$, so that $x_0$ is fixed by every element of $H$.  
Define the $H$-equivariant map 
$i_2\colon X \longrightarrow X\times X$ as $i_2=(c_{x_0},\,id_X).$
Note that $\ct_H(X)=\D_H(id_X,c_{x_0})$.  
By Propositions \ref{prop:composition-inequalities} and \ref{prop:subgroup inequality}, we obtain
\[
\ct_H(X)=\D_H(pr_2^X\circ i_2,\,pr_1^X\circ i_2)
        \;\leq\;\D_H(pr_1^X,pr_2^X)
        \;\leq\;\D_G(pr_1^X,pr_2^X)
        =\TC_G(X).
\]
\end{proof}

We conclude this section by establishing the equivariant homotopy invariance of the equivariant homotopic distance.

The proof of the following result is analogous to the nonequivariant case, so we prefer to omit it.
\begin{proposition}\label{prop: left right hootopy inverse}
Let $f,g\colon X \to Y$ be $G$-maps.  
\begin{enumerate}
    \item If there exists a $G$-map $h\colon Y \to Y'$ with a left $G$-homotopy inverse, then
    \[
        \D_G(h\circ f,\,h\circ g)\;=\;\D_G(f,g).
    \]
    \item If there exists a $G$-map $h\colon X' \to X$ with a right $G$-homotopy inverse, then
    \[
        \D_G(f\circ h,\,g\circ h)\;=\;\D_G(f,g).
    \]
\end{enumerate}
\end{proposition}

This shows that the equivariant homotopic distance is an equivariant homotopy invariant in the following sense:

\begin{corollary}\label{invar}
Let $f,g\colon X \to Y$ and $f',g'\colon X' \to Y'$ be $G$-maps.  
If there exist $G$-homotopy equivalences $\alpha\colon Y \to Y'$ and $\beta\colon X' \to X$ such that 
$\alpha\circ f\circ \beta \;\simeq_G\; f'$ and $\alpha\circ g\circ \beta \;\simeq_G\; g',$
then $\D_G(f,g)\;=\;\D_G(f',g').$
\end{corollary}

\begin{corollary}
If $f\colon X \to Y$ is a $G$-homotopy equivalence, then
\[
    \TC_G(X)=\TC_G(Y)
    \qquad\text{and}\qquad
    \ct_G(X)=\ct_G(Y).
\]
\end{corollary}

\section{The triangle inequality and its consequences}\label{sec: triangle inequality}

Our goal now is to show that the equivariant homotopic distance satisfies the triangle inequality under some mild assumptions on the domain space $X$. This result will enable us to endow the set of equivariant homotopy classes $[X, Y]_G$ with a metric structure. Unlike the argument developed by E. Macías-Virgós and D. Mosquera-Lois in \cite{macias2022homotopic}---which relies heavily on a remarkable result by J. Oprea and J. Strom \cite[Lemma 4.3]{oprea2011mixing}---our approach follows a different path. Specifically, we avoid arguments based on open covers and instead employ categorical techniques.

\medskip
Consider $G$-maps $f,g,h:X\to Y$. We know that, from the following pullbacks
$$\xymatrix{
{\mathcal{P}(f,g)} \ar[r] \ar[d]_{\overline{\pi }_Y^{f,g}} & {Y^I} \ar[d]^{\pi _Y} &  {\mathcal{P}(f,h)} \ar[r] \ar[d]_{\overline{\pi }_Y^{f,h}} & {Y^I} \ar[d]^{\pi _Y} & {\mathcal{P}(h,g)} \ar[r] \ar[d]_{\overline{\pi }_Y^{h,g}} & {Y^I} \ar[d]^{\pi _Y}  \\
{X} \ar[r]_{(f,g)} & {Y\times Y} & {X} \ar[r]_{(f,h)} & {Y\times Y} & {X} \ar[r]_{(h,g)} & {Y\times Y}  }
$$
\noindent we have $\D _G(f,g)=\sct _G(\overline{\pi }_Y^{f,g}),$  $\D _G(f,h)=\sct _G(\overline{\pi }_Y^{f,h}),$ and $\D _G(h,g)=\sct _G(\overline{\pi }_Y^{h,g}).$

\medskip
We begin with the following lemma. For any product of $G$-spaces, we will consider the diagonal action of $G$.

\begin{lemma}\label{lem-triang-1}
Let $Y$ be a $G$-space and consider $$Y^I\times _Y Y^I:=\{(\alpha ,\beta )\in Y^I\times Y^I \mid \alpha (1)=\beta (0)\}$$ and take
$(Y\times Y)\times_Y (Y\times Y):=\{(y_1,y_2,y_3,y_4)\in Y\times Y\times Y\times Y \mid y_2=y_3\}$.
Then, the following diagram is a pullback of $G$-spaces and $G$-maps:
$$\xymatrix{
{Y^I\times_Y Y^I} \ar[d]_{\pi _Y\times _Y \pi _Y} \ar@{^(->}[rr] & & {Y^I\times Y^I} \ar[d]^{\pi _Y\times \pi _Y} \\
{(Y\times Y)\times_Y (Y\times Y)} \ar@{^(->}[rr] & & {Y\times Y\times Y\times Y.}
}$$
Here, $\pi _Y\times _Y \pi _Y:Y^I\times_Y Y^I\to (Y\times Y)\times _Y(Y\times Y)$ denotes the obvious restriction of $\pi _Y\times \pi _Y$. Therefore, $\pi _Y\times _Y \pi _Y$ is a $G$-fibration.    
\end{lemma}

\begin{proof}
Since $(Y\times Y)\times _Y(Y\times Y)$ is a $G$-subspace of $Y\times Y\times Y\times Y$, then this pullback is, up to $G$-homeomorphism, the preimage
$$(\pi _Y\times  \pi _Y)^{-1}((Y\times Y)\times _Y(Y\times Y))=Y^I\times _YY ^I.$$  
\end{proof}

Now we define the following pullback of $G$-maps:
$$\xymatrix{
{Q} \ar[d]_q \ar[rr] & & {Y^I\times _Y Y^I} \ar[d]^{\pi _Y\times _Y \pi _Y} \\
{X} \ar[rr]_(0.35){(f,h,h,g)} & & {(Y\times Y)\times _Y(Y\times Y)} .}$$

\begin{remark}\label{triang-remark}
As the composite of pullbacks is also a pullback, composing with the pullback given in Lemma \ref{lem-triang-1} above, we also have a pullback
$$\xymatrix{
{Q} \ar[d]_q \ar[rr] & & {Y^I\times Y^I} \ar[d]^{\pi _Y\times \pi _Y} \\
{X} \ar[rr]_(0.35){(f,h,h,g)} & & {Y\times Y\times Y\times Y} .}$$
\end{remark}

\begin{lemma}\label{lem-triang-2}
There is a pullback of $G$-maps
$$\xymatrix{
{Q} \ar[rr] \ar[d]_q & & {\mathcal{P}(f,h)\times \mathcal{P}(h,g)} \ar[d]^{\overline{\pi }_Y^{f,h}\times \overline{\pi }_Y^{h,g}} \\
{X} \ar[rr]_{\Delta _X} & & {X\times X} .}
$$ \noindent where $\Delta _X$ denotes the diagonal map. In particular, $\sct _G(q)\leq \sct _G(\overline{\pi }_Y^{f,h}\times \overline{\pi }_Y^{h,g})$.
\end{lemma}

\begin{proof}
Taking into account that the right square in the following diagram is a pullback, by the universal property of pullbacks, there is an induced map making commutative the left square:
$$\xymatrix{
{Q} \ar[d]_q \ar@{.>}[rr] & & {\mathcal{P}(f,h)\times \mathcal{P}(h,g)} \ar[d]^{\overline{\pi }_Y^{f,h}\times \overline{\pi }_Y^{h,g}} \ar[rr] & & {Y^I\times Y^I} \ar[d]^{\pi _Y\times \pi _Y} \\
{X} \ar[rr]_{\Delta _X} & & {X\times X} \ar[rr]_{(f,h)\times (h,g)} & & {Y\times Y\times Y\times Y.}
}$$
However, since the composite of these two squares is a pullback (observe that $((f,h)\times (h,g))\circ \Delta _X=(f,h,h,g)$ and consider Remark \ref{triang-remark}) we have that the left square is necessarily a pullback. The inequality $\sct _G(q)\leq \sct _G(\overline{\pi }_Y^{f,h}\times \overline{\pi }_Y^{h,g})$ comes from \cite[Proposition 4.3]{EqTC}.
\end{proof}

The final step is provided by the following lemma:

\begin{lemma}\label{lem-triang-3}
There exists a commutative triangle relating $q$ with $\overline{\pi }_Y^{f,g}$:
$$\xymatrix{
{Q} \ar[rr] \ar[dr]_q & & {P(f,g)} \ar[dl]^{\overline{\pi }_Y^{f,g}} \\
 & {X.} & }$$ In particular, $\sct _G(\overline{\pi }_Y^{f,g})\leq \sct _G(q)$.
\end{lemma}

\begin{proof}
Consider the following commutative diagram of $G$-maps:
$$\xymatrix{
{X} \ar@{=}[d] \ar[rr]^(0.35){(f,h,h,g)} & &  {(Y\times Y)\times _Y(Y\times Y)} \ar[d]^t & & {Y^I\times _Y Y^I} \ar[ll]_(0.4){\pi _Y \times _Y \pi _Y} \ar[d]^m \\
{X} \ar[rr]_{(f,g)} & & {Y\times Y} & & {Y^I.} \ar[ll]^{\pi _Y} }$$
Here, $t:(Y\times Y)\times _Y(Y\times Y)\to Y\times Y$ is defined as $t(y_1,y,y,y_2):=(y_1,y_2)$, and $m:Y^I\times _Y Y^I\to Y^I$ is defined as
$$m(\alpha ,\beta ):=\alpha *\beta $$ \noindent that is, the concatenation of paths in $Y.$ Observe that both $t$ and $m$ are $G$-map. Using this diagram and the universal property of pullbacks, one can readily verify the existence of a $G$-map $Q \to \mathcal{P}(f,g)$ between the pullbacks of the horizontal arrows, making the resulting cube diagram commutative. In particular, this $G$-map guarantees the commutativity of the triangle described in the lemma. The inequality $\sct_G(\overline{\pi}_Y^{f,g}) \leq \sct_G(q)$ then follows directly from Lemma~\ref{ineq}.
\end{proof}

Now we are in a position to prove the triangle inequality:

\begin{theorem}\label{triang}
Let $f,g,h:X\to Y$ be $G$-maps where $X\times X$ is completely normal and $X$ is Hausdorff (for instance, if $X$ is metrizable). Then
$$\D_G(f,g)\leq \D_G(f,h)+\D_G(h,g).$$
\end{theorem}

\begin{proof}
By lemmas \ref{lem-triang-3} and \ref{lem-triang-2} above one obtains the inequalities
$$\D _G(f,g)=\sct _G(\overline{\pi }_Y^{f,g})\leq \mbox{secat}(q)\leq \sct _G(\overline{\pi }_Y^{f,h}\times \overline{\pi }_Y^{h,g}).$$
Moreover, since $X\times X$ is completely normal, by Proposition \ref{sub} we have:  $$\sct _G(\overline{\pi }_Y^{f,h}\times \overline{\pi }_Y^{h,g})\leq \sct _G(\overline{\pi }_Y^{f,h})+\sct _G(\overline{\pi }_Y^{h,g})=\D _G(f,h)+\D _G(h,g).\qedhere$$
\end{proof}

\begin{remark}
Observe that the arguments presented here remain valid in the nonequivariant setting. Therefore, in view of \cite[Theorem 2.2]{MR2199456} and \cite[Remark 7.1]{MR4045098}, we recover the triangle inequality established by E.~Macías-Virgós and D.~Mosquera-Lois under the sole assumption that $X$ is a normal space. We emphasize that our alternative approach does not rely on the lemma of J.~Oprea and J.~Strom \cite[Lemma~4.3]{oprea2011mixing}. We expect that the same reasoning may extend to other settings, such as Quillen model categories \cite{MR3027356}, whenever the abstract notion of sectional category satisfies the subadditivity property.
\end{remark}

As a consequence of this result, we can refine the inequality stated in Corollary~\ref{ineq-D-cat}.

\begin{corollary}
Let $f,g:X\to Y$ be $G$-maps, where $X\times X$ is completely normal, $X$ is Hausdorff, and $Y$ is a $G$-connected space with $Y^G\neq \emptyset$. Then
$$\D_G(f,g)\leq \ct_G(f)+\ct_G(g).$$
\end{corollary}

\begin{proof}
Let $y_0\in Y^G$ be a fixed point, and denote by $c_{y_0}:X\to Y$ the constant map at $y_0$.  
Applying Proposition~\ref{ct-mor-car} and Theorem~\ref{triang}, it follows that
$$\D_G(f,g)\leq \D_G(f,c_{y_0})+\D_G(c_{y_0},g)=\ct_G(f)+\ct_G(g).$$
\end{proof}

We now establish a composition inequality:

\begin{corollary}
Let $f,g:X\to Y$ and $f',g':Y\to Z$ be $G$-maps, where $X\times X$ is completely normal and $X$ is Hausdorff. Then
$$\D_G(f'\circ f,g'\circ g)\leq \D_G(f,g)+\D_G(f',g').$$   
\end{corollary}

\begin{proof}
Applying Theorem \ref{triang} followed by Proposition \ref{prop:composition-inequalities}, it follows that
$$\D_G(f'\circ f,g'\circ g)\leq \D_G(f'\circ f,f'\circ g)+\D_G(f'\circ g,g'\circ g)\leq \D_G(f,g)+\D_G(f',g').$$
\end{proof}

We now prove a product inequality for the equivariant homotopic distance. 
Given $G$-maps $f,g:X\to Y$ and $h:X'\to Y'$, we have
$$\D_G(f\times h, g\times h) \leq \D_G(f,g),$$
since $f|_U\simeq_G g|_U$ implies $(f\times h)|_{U\times X'}\simeq_G(g\times h)|_{U\times X'}$ for every $G$-invariant open subset $U\subseteq X$. 
Similarly, one has
$\D_G(h\times f, h\times g)\leq \D_G(f,g).$

Under additional hypotheses, equality can also be obtained. 
Suppose that $X'^G\neq\emptyset$ and fix $x'_0\in X'^G$. 
Consider the $G$-map $i_1:X\to X\times X'$ defined by $i_1(x)=(x,x'_0)$. 
If $pr_1:Y\times Y'\to Y$ denotes the projection onto the first factor, then by applying Proposition~\ref{prop:composition-inequalities} twice we get
$$\D_G(f,g)=\D_G(pr_1\circ(f\times h)\circ i_1,\,pr_1\circ(g\times h)\circ i_1)\leq \D_G(f\times h,g\times h).$$

Building on this observation, we now establish the general product inequality.

\begin{proposition}\label{product}
Consider $G$-maps $f,g:X\to Y$ and $f',g':X'\to Y'$. 
If $X$ and $X'$ are metrizable, then
$$\D_G(f\times f', g\times g')\leq \D_G(f,g)+\D_G(f',g').$$
\end{proposition}

\begin{proof}
Applying Theorem~\ref{triang} together with the preceding remarks, it follows that
$$\D_G(f\times f', g\times g')\leq \D_G(f\times f', g\times f')+\D_G(g\times f', g\times g')\leq \D_G(f,g)+\D_G(f',g').$$
\end{proof}

Bayeh and Sarkar \cite[Proposition~2.10]{B-S} established the product inequality for the equivariant LS category. As a consequence of Proposition~\ref{product}, we recover this result as a special case:

\begin{corollary}
Let $X$ and $X'$ be metrizable $G$-spaces. 
If $X$ and $X'$ are $G$-connected and $X^G\neq\emptyset\neq X'^G$, then
$$\ct_G(X\times X')\leq \ct_G(X)+\ct_G(X').$$
\end{corollary}

\begin{proof}
Apply Proposition~\ref{product} and Proposition~\ref{ct-mor-car} to 
$f=id_X$, $f'=id_{X'}$, $g=c_{x_0}$, and $g'=c_{x'_0}$, 
where $x_0\in X^G$ and $x'_0\in X'^G$.    
\end{proof}

For a compact Lie group $G$ acting on smooth $G$-manifolds, the product inequality for equivariant topological complexity was proved by González and Grant in \cite[Theorem~4.2]{G-G}. 
We extend this result to arbitrary metrizable $G$-spaces:

\begin{corollary}
Let $X$ and $X'$ be metrizable $G$-spaces. Then
$$\TC_G(X\times X')\leq \TC_G(X)+\TC_G(X').$$   
\end{corollary}

\begin{proof}
Consider $f=pr_1^X:X\times X\to X$, $f'=pr_1^{X'}:X'\times X'\to X'$, $g=pr_2^X:X\times X\to X$, and $g'=pr_2^{X'}:X'\times X'\to X'$. 
There is a natural $G$-homeomorphism making the following diagram commute:
$$
\xymatrix{
{(X\times X')\times(X\times X')} \ar[rr]^{\cong} \ar[dr]_{pr_\varepsilon^{X\times X'}} & & {(X\times X)\times(X'\times X')} \ar[dl]^{pr_\varepsilon^X\times pr_\varepsilon^{X'}} \\
 & {X\times X'} & 
}
$$
for all $\varepsilon\in\{1,2\}$. 
Therefore,
$$\D_G(pr_{1}^{X\times X'}, pr_{2}^{X\times X'})=\D_G(pr_1^X\times pr_1^{X'},\,pr_2^X\times pr_2^{X'}).$$
Applying Proposition~\ref{product}, the desired inequality follows.
\end{proof}

\begin{remark}
The results of this section might also be obtained—possibly under weaker assumptions on the domain space $X$—by adapting to the equivariant setting the lemma of J.~Oprea and J.~Strom \cite[Lemma 4.3]{oprea2011mixing}, used in the proof of the triangle inequality by E.~Macías-Virgós and D.~Mosquera-Lois. 
We have instead presented an alternative approach, which may be applicable in broader contexts beyond the equivariant case.
\end{remark}

\section{Hopf $G$-spaces}\label{sec: Hopg G-space}

Equivariant analogues of Hopf spaces, known as \emph{Hopf $G$-spaces}, are a classical object of study in equivariant homotopy theory. One of the earliest systematic treatments of these spaces can be found in Bredon’s lecture notes \cite{Bredon}.  

In this section, we recall the definition of Hopf $G$-spaces and explore their connection with the equivariant homotopic distance.  

\begin{definition}
A \emph{Hopf $G$-space} is a pointed $G$-space $(X,x_0)$ with $x_0\in X^G$, together with a pointed $G$-map (the \emph{equivariant multiplication})
\[
  \mu : X \times X \longrightarrow X, 
  \qquad (x,y)\mapsto x\cdot y.
\]
The multiplication satisfies unit conditions: there exist pointed $G$-homotopies
\[
  \mu \circ i_1 \;\simeq_G^*\; id_X,
  \qquad
  \mu \circ i_2 \;\simeq_G^*\; id_X,
\]
where $i_1,i_2:X\to X\times X$ are defined by $i_1(x)=(x,x_0)$ and $i_2(x)=(x_0,x)$.  

The space $X$ is said to be \emph{$G$-homotopy associative} if, in addition, there exists a pointed $G$-homotopy between the two natural compositions of the multiplication:
\[
  \mu \circ (\mu \times id_X) \;\simeq_G^*\; 
  \mu \circ (id_X \times \mu).
\]
\end{definition}

\begin{remark}
Throughout this section, the notation $\simeq_G$ denotes $G$-equivariant homotopy, while $\simeq_G^*$ stands for \emph{pointed} $G$-equivariant homotopy, i.e. homotopies through $G$-maps that fix the chosen basepoint.
\end{remark}

A \emph{$G$-homotopy right inverse} of an equivariant multiplication $\mu$ on a pointed $G$-space $X$ is a pointed $G$-map $u:X\to X$ such that
\[
  \mu\circ (id_X,u)\;\simeq_G^*\; s_X\circ p_X,
\]
where $p_X:X\to *$ is the projection to the basepoint and $s_X:* \to X$ is the inclusion.  
An analogous definition applies for $G$-homotopy left inverses, replacing $(id_X,u)$ with $(u,id_X)$.

If $\mu$ is $G$-homotopy associative, then any pointed $G$-map that is a $G$-homotopy right inverse is automatically also a $G$-homotopy left inverse. In this case one simply speaks of a \emph{$G$-homotopy inverse}, customarily denoted by $u(x)=x^{-1}$.

A \emph{group-like $G$-space} is, by definition, a $G$-homotopy associative Hopf $G$-space whose multiplication admits a $G$-homotopy inverse.  
As in the classical case (cf.~\cite[page~119]{whitehead2012elements}), one readily verifies that a $G$-homotopy associative Hopf $G$-space $X$ is group-like if and only if the \emph{equivariant shearing map}
\[
  sh:X\times X\longrightarrow X\times X, 
  \qquad (x,y)\mapsto (x,x\cdot y),
\]
is a pointed $G$-homotopy equivalence.

\bigskip
We now focus on Hopf $G$-spaces that admit a division $G$-map.

\begin{definition}
Let $X$ be a Hopf $G$-space.  
A \emph{division $G$-map} is a pointed $G$-map $\delta:X\times X\to X$ such that there exists a pointed $G$-homotopy
\[
  \mu \circ (pr_1^X,\delta)\;\simeq_G^*\;pr_2^X,
\]
where $pr_1^X,pr_2^X:X\times X\to X$ are the canonical projections.
\end{definition}

\begin{lemma}\label{division}
Let $X$ be a Hopf $G$-space. Then $X$ admits a division $G$-map if and only if the equivariant shearing map
\[
  sh=(pr_1^X,\mu):X\times X\longrightarrow X\times X
\]
admits a pointed $G$-homotopy right inverse.
\end{lemma}

\begin{proof}
Suppose $\delta:X\times X\to X$ is a division $G$-map, and set $\psi:=(pr_1^X,\delta)$. Then
\[
  sh\circ\psi
   =(pr_1^X,\mu)\circ(pr_1^X,\delta)
   =(pr_1^X,\mu\circ(pr_1^X,\delta))
   \;\simeq_G^*\;(pr_1^X,pr_2^X)
   =id_{X\times X}.
\]

Conversely, suppose $\psi:X\times X\to X\times X$ is a pointed $G$-homotopy right inverse of $sh$.  
Write $\psi=(\psi_1,\psi_2)$ and define $\delta:=\psi_2$.  
Then
\[
  (pr_1^X,pr_2^X)=id_{X\times X}
   \;\simeq_G^*\;sh\circ\psi
   =(pr_1^X,\mu)\circ(\psi_1,\psi_2)
   =(\psi_1,\mu\circ(\psi_1,\delta)).
\]
Comparing components yields $pr_1^X\simeq_G^*\psi_1$ and $pr_2^X\simeq_G^*\mu\circ(pr_1^X,\delta)$, which shows that $\delta$ is indeed a division $G$-map.
\end{proof}

If $X$ is a group-like $G$-space, then—as previously noted—the shearing $G$-map is a pointed $G$-homotopy equivalence. Consequently, by Lemma~\ref{division}, the space $X$ admits a division $G$-map $\delta: X \times X \to X$ defined by $\delta(x, y) = x^{-1} \cdot y.$
However, a Hopf $G$-space does not necessarily need to be group-like in order to admit a division $G$-map. In the following, we explore sufficient conditions for the existence of such a $G$-map.

\begin{proposition}
Let $X$ be a $G$-connected Hopf $G$-space endowed with the structure of a finite $G$-CW complex. Then $X$ admits a division $G$-map.    
\end{proposition}

\begin{proof}
By Lemma \ref{division}, it suffices to prove that the equivariant shearing map
\[
  sh:X\times X \longrightarrow X\times X
\]
is a pointed $G$-homotopy equivalence. 
For any closed subgroup $H\leq G$, the fixed point space $X^H$ inherits the structure of a classical Hopf space with multiplication $\mu^H$. 
The induced map on fixed points,
\[
  sh^H:X^H\times X^H \longrightarrow X^H\times X^H,\qquad (x,y)\mapsto(x,x\cdot y),
\]
is precisely the classical shear map on $X^H$. 
As the shear map of any path-connected Hopf space is a weak homotopy equivalence, it follows that $sh^H$ is a weak equivalence for every closed subgroup $H\leq G$. 
Therefore $sh$ is an equivariant weak equivalence. 
Since $X\times X$ is a $G$-CW complex, the equivariant Whitehead theorem implies that $sh$ is in fact a $G$-homotopy equivalence.
Moreover, by \cite[Proposition 8.12]{MR889050} and \cite[Proposition 2.7]{G-symmetrized}, the $G$-space $X\times X$ is a $G$-ENR, hence a well-based $G$-space. 
We conclude that $sh$ is a pointed $G$-homotopy equivalence.
\end{proof}

We now establish a connection between the equivariant homotopic distance and the equivariant Lusternik–Schnirelmann category in the setting of Hopf $G$-spaces with a division $G$-map.

\begin{theorem}\label{thm:phd-leq-pcat}
Let $X$ be a $G$-connected Hopf $G$-space with a division $G$-map, and let $f,g:X\times X\rightarrow X$ be pointed $G$-maps. Then, we have the inequality $\D_G(f,g)\leq \ct_G(X).$
\end{theorem}

\begin{proof} 
First observe that, by Proposition \ref{ct-mor-car}, we have 
$$\ct_G(X)=\ct _G(id_X)=\D _G(id_X,c_{x_0}),$$ 
where $x_0$ is the base point of $X.$ 
If we consider $U$ to be a $G$-categorical open subset of $X$ (i.e., $inc _U\simeq _G c_{x_0}|_U$), the division $G$-map $\delta: X \times X \to X$, and the composite map $\phi := \delta \circ (f, g): X \times X \to X,$
along with the invariant open subset $V := \phi^{-1}(U)$, then we obtain the strictly commutative diagram
$$
\xymatrix{
{V} \ar[d]_{\phi|_V} \ar@{^{(}->}[rr]^{\mbox{inc}_V} & & {X \times X} \ar[d]^{\phi} \\
{U} \ar@{^{(}->}[rr]_{\mbox{inc}_U} & & {X}.
}
$$

This gives rise to the following sequence of pointed $G$-homotopies:
$$
\begin{array}{lcl}
g|_V &=& g \circ \mbox{inc}_V \\
     &=& pr_2^X \circ (f, g) \circ \mbox{inc}_V \\
     &\simeq^*_G& \mu \circ (pr_1^X, \delta) \circ (f, g) \circ \mbox{inc}_V \\
     &=& \mu \circ (f, \delta \circ (f, g)) \circ \mbox{inc}_V \\
     &=& \mu \circ (f, \phi) \circ \mbox{inc}_V \\
     &=& \mu \circ (f \circ \mbox{inc}_V, \phi \circ \mbox{inc}_V) \\
     &=& \mu \circ (f|_V, \mbox{inc}_U \circ \phi|_V) \\
     &\simeq_G& \mu \circ (f|_V, c_{x_0}|_U \circ \phi|_V) \\
     &=& \mu \circ (f|_V,c_{x_0}\circ f|_V) \\
     &=& \mu \circ (id_X, c_{x_0}) \circ f|_V \\
     &\simeq^*_G& id_X \circ f|_V = f|_V.
\end{array}
$$
Thus, we have $g|_{V} \simeq_G f|_{V}$. By applying this argument to invariant open covers, we derive the desired inequality.
\end{proof}

Lupton and Scherer \cite[Theorem 1]{L-S} showed that the topological complexity of a path-connected CW $H$-space coincides with its Lusternik–Schnirelmann category. The following result establishes an equivariant analogue of this fact.

\begin{corollary}\label{cor: TCG equals catG}
Let $X$ be a $G$-connected Hopf $G$-space with a division $G$-map. Then
$$\TC_G(X) = \ct_G(X).$$ 
\end{corollary}

\begin{proof}
The inequality $ \TC_G(X) \leq \ct_G(X) $ follows by taking $ f = pr_1 $ and $ g = pr_2 $ in \Cref{thm:phd-leq-pcat}. The reverse inequality follows from Corollary~\ref{consecu}(2).
\end{proof}

In the following result, we use the product of two $G$-maps $f, g : X \to Y$, defined by
$$
f \cdot g := \mu \circ (f, g) = \mu \circ (f \times g) \circ \Delta,
$$
where $Y$ is a Hopf $G$-space. Here, $\Delta : X \to X \times X$ denotes the diagonal map, $f \times g : X \times X \to Y \times Y$ is the product $G$-map, and $\mu : Y \times Y \to Y$ is the equivariant multiplication arising from the Hopf $G$-space structure on $Y$.

\begin{proposition}\label{prop:multiplication-fib-spaces-ineq}
Let $f, g, h : X \to Y$ be $G$-maps, where $Y$ is a Hopf $G$-space. Then
$$
\D_G(f \cdot h, g \cdot h) \leq \D_G(f, g).
$$
\end{proposition}

\begin{proof}
It is known that $\D_G(f \times h, g \times h) \leq \D_G(f, g)$. Applying Proposition \ref{prop:composition-inequalities}, we obtain
$$
\begin{aligned}
\D_G(f \cdot h, g \cdot h) 
&= \D_G\big(\mu \circ (f \times h) \circ \Delta,\, \mu \circ (g \times h) \circ \Delta\big) \\
&\leq \D_G(f \times h, g \times h) \\
&\leq \D_G(f, g),
\end{aligned}
$$
as desired.
\end{proof}

\begin{corollary}\label{cor:grplike-space}
If $X$ is a  group-like $G$-space, then $\D_G(\mu ,\delta )=\D_G(id_X,u),$ \noindent where $\delta :X\times X\rightarrow X$ is the division $G$-map, and $u:X\rightarrow X$ represents the equivariant homotopy inverse.
\end{corollary}

\begin{proof}
Recall the pointed $G$-map $i_1:X \to X \times X$ defined by $i_1(x) = (x, x_0)$. Since $\mu \circ i_1 \simeq ^*_G id_X$ and $\delta \circ i_1 \simeq ^*_G u$, it follows that 
$$
\D_G(id_X, u) = \D_G(\mu \circ i_1, \delta \circ i_1) \leq \D_G(\mu, \delta).
$$

Conversely, observe that $\mu = pr_1^X \cdot pr_2^X = (id_X \circ pr_1^X) \cdot pr_2^X$ and $\delta = (u \circ pr_1^X) \cdot pr_2^X$. Therefore,
$$
\D_G(\mu, \delta) \leq \D_G(id_X \circ pr_1^X, u \circ pr_1^X) \leq \D_G(id_X, u). 
$$
\end{proof}

Now, assuming that $X$ is metrizable, we generalize the inequality stated in \Cref{prop:multiplication-fib-spaces-ineq}.

\begin{proposition}\label{prop:fib-mult-ineq}
Let $f,g,h,h':X\rightarrow Y$ be $G$-maps, where $Y$ is a Hopf $G$-space and $X$ is metrizable. Then
$\D_G(f\cdot h,g\cdot h')\leq \D_G(f,g) + \D_G(h,h').$
\end{proposition}

\begin{proof}
Using the triangle inequality, the fact $\D_G(f\times h,g\times h)\leq \D_G(f,g)$ and \Cref{prop:composition-inequalities}, we obtain
$$\begin{array}{ll}
\D_G(f\cdot h,g\cdot h')& = \D_G(\mu \circ (f\times h)\circ \Delta ,\mu \circ (g\times h')\circ \Delta )\\
& \leq \D_G(f\times h,g\times h')\\
&\leq \D_G(f\times h, g\times h)+ \D_G(g\times h, g\times h')\\
&\leq \D_G(f,g)+ \D_G(h,h'). 
\end{array}
$$
\end{proof}

\subsection{Hopf $G$-spheres}\label{Hopf G sphere}
We now apply our results to the case of spheres equipped with Hopf $G$-structures.
It is a well-known theorem of Adams that $S^n$ is a Hopf space if and only if $n = 0, 1, 3, 7$. 
Note that the Hopf structures on $S^1$, $S^3$, and $S^7$ are given by the multiplication in $\mathbb{C}$ (the complex numbers), $\mathbb{H}$ (the quaternions), and $\mathbb{O}$ (the octonions), respectively (see \cite{I}). 
The corresponding automorphism groups of $\mathbb{C}$, $\mathbb{H}$, and $\mathbb{O}$ are $\mathbb{Z}_2$, $\mathrm{SO}(3)$, and $G_2$, respectively. 
The action of these automorphism groups preserves the Hopf structures, making $S^1$, $S^3$, and $S^7$ into Hopf $G$-spaces, where $G$ is the corresponding automorphism group (see Proposition 2.1 of \cite{C-C}).

Note that the automorphism group of $\mathbb{C}$ is generated by the complex conjugation involution. 
Therefore, $(S^1)^G = \{1, -1\}$, and thus $S^1$ is not a $G$-connected Hopf $G$-space. 
Consequently, $\TC_{G}(S^1)=\infty$.

We now examine how $\mathrm{SO}(3)$ acts on $S^3$ as the automorphism group of the quaternions. 
Recall that
\[
S^3 = \{z = a + bi + cj + dk \mid \|z\| = 1 \},
\]
where $i, j, k$ are the imaginary quaternions satisfying $i^2=j^2=k^2=ijk=-1$. 
The imaginary quaternions $i, j, k$ span a vector space isomorphic to $\mathbb{R}^3$, denoted by $\mathrm{Im}(\mathbb{H})$. 
An element $v = b i + c j + d k$ can thus be viewed as a vector in $\mathrm{Im}(\mathbb{H}) \cong \mathbb{R}^3$.

Let $g \in \mathrm{SO}(3)$ and $z = a + v$. 
Then $g \cdot z = a + g(v)$, where $g(v)$ denotes the standard action of $\mathrm{SO}(3)$ on $\mathbb{R}^3$ via rotations. 
Clearly, under this action, the real part $a$ is fixed. 
In particular, $(S^3)^G = \{1, -1\}$, and hence $\TC_G(S^3)=\infty$.

However, we can consider subgroups $K \subseteq \mathrm{SO}(3)$ that may lead to a $K$-connected Hopf-$K$ space structure on $S^3$. 
Since the action of $\mathrm{SO}(3)$ on $\mathbb{H}$ is linear, the fixed set $\mathbb{H}^K$ for any closed subgroup $K$ of $\mathrm{SO}(3)$ is a real subspace. 
The corresponding induced action on $S^3$ then has fixed sets that are spheres $S^r$ for some $r \leq 3$.

For example, if $K=\mathrm{SO}(2)$ acts on $S^3$ via rotations fixing the real axis $i \cdot \mathbb{R}$, then
\[
(S^3)^{K}=\{a+b i 
\in \mathbb{H} \mid \|a+b i\|=1 \} \cong S^1.
\]
Since the action of any closed subgroup $H$ of $K$ is still linear and via rotations, we also have $(S^3)^H\cong S^1$. 
Therefore, in this case $S^3$ becomes a $K$-connected Hopf $K$-space. 
The principal orbit of this action is homeomorphic to $S^1$. 
Hence, using \cite[Theorem \RomanNumeralCaps 4.3.8]{Br1}, we have $\dim(S^3/K)=\dim(S^3)-1=2$. 
Then, by \Cref{cor: TCG equals catG} and \cite[Corollary 1.12]{Eqlscategory}, we obtain
\[
1 \leq \TC_{SO(2)}(S^3)=\ct_{SO(2)}(S^3) \leq 2.
\]
If $K'$ is any finite subgroup of $K=SO(2)$, then it follows from \cite[Proposition 3.1]{Z-K} that $\ct_{K'}(S^3)=1$. 
Thus, from \Cref{cor: TCG equals catG}, it follows that
\[
\TC_{K'}(S^3)=1.
\]

More generally, if $K$ is any positive-dimensional closed subgroup of $\mathrm{SO}(3)$ whose fixed point set in $\mathrm{Im}(\mathbb{H})$ is nonempty, then $S^3$ is $K$-connected. 
Due to the linearity of the action, the fixed set in this case is always homeomorphic to $S^1$. 
Therefore, again using \Cref{cor: TCG equals catG} and \cite[Corollary 1.12]{Eqlscategory},
\[
1 \leq \TC_{K}(S^3)=\ct_{K}(S^3) \leq 2.
\]

Similar to the quaternion case, the group $G_2 \subseteq \mathrm{SO}(7)$ acts on the octonions via its action on the purely imaginary octonions. 
We consider the octonion algebra $\mathbb{O}$ as $\mathbb{C}\oplus \mathbb{C}^3$, where the $\mathbb{C}^3$ part consists of imaginary octonions. 
Again, this action is linear and orthogonal, inducing an action on $S^7$ viewed as the unit octonions. 
Clearly, the fixed set $(S^7)^{G_2}\cong S^1$.

It is known that the special unitary group $\mathrm{SU}(3)$ embeds into $G_2$. 
We also have $(S^7)^{\mathrm{SU}(3)}\cong S^1$, since the $\mathrm{SU}(3)$ action on the imaginary octonions has no fixed points outside $\mathbb{C}$. 
Since this action is linear, the fixed sets are spheres $S^r$ with $r\geq 1$. 
Therefore, $S^7$ is $\mathrm{SU}(3)$-connected. 
Moreover, it can be observed that $\mathrm{SU}(3)$ acts transitively on the unit sphere in the imaginary octonions. 
Therefore, the principal orbit of an imaginary point in $S^7$ is $S^5 \subseteq \mathbb{C}^3$, giving $\dim(S^7/\mathrm{SU}(3))=7-5=2$. 
Hence, using \Cref{cor: TCG equals catG} and \cite[Corollary 1.12]{Eqlscategory},
\[
1 \leq \TC_{\mathrm{SU}(3)}(S^7)=\ct_{\mathrm{SU}(3)}(S^7) \leq 2.
\]

Finally, if $G$ is any finite subgroup of $SU(3)$ acting on $S^7$, then $S^7$ is $G$-connected. 
By \cite[Proposition 3.1]{Z-K}, we have $\ct_{G}(S^7)=1$. 
Then, from \Cref{cor: TCG equals catG}, it follows that
\[
\TC_{G}(S^7)=1.
\]

\section{Equivariant fibrations}\label{sec: equivariant fibrations}
In this section, we estimate the equivariant homotopic distance between fibre-preserving $G$-maps of $ G$-fibrations in terms of the induced distance on the fibres and the equivariant Lusternik–Schnirelmann category of the base. 

Suppose $p:E\to B$ and $p':E'\to B'$ are $G$-fibrations, with $B$ and $B'$ $G$-connected, and let $f,g:E\to E'$ and $\bar f,\bar g:B\to B'$ be $G$-maps satisfying $p'\circ f=\bar f\circ p$ and $p'\circ g=\bar g\circ p$. This gives the commutative diagram
\begin{equation}\label{dia:comm-diagram-G-fibration}
\xymatrix{
E \ar[d]_{p} \ar[rr]^{f,g} && E' \ar[d]^{p'} \\
B \ar[rr]_{\bar f,\bar g} && B'.}
\end{equation}
Choose $b_0\in B^G$ with $\bar f(b_0)=\bar g(b_0)=b'_0\in B'^G$. Then $f$ and $g$ restrict to $G$-maps $f_0,g_0:F\to F'$, where $F=p^{-1}(b_0)$ and $F'=(p')^{-1}(b'_0)$.  

\begin{theorem}\label{thm: G-fibration}
For $f,g:E\to E'$ as above,
\[
D_G(f,g)+1 \;\leq\; (D_G(f_0,g_0)+1)\,(\ct_G(B)+1).
\]
\end{theorem}
\begin{proof}
The proof is similar to that of \cite[Theorem 6.1]{macias2022homotopic}.   
\end{proof}

We now turn to applications of \Cref{thm: G-fibration}.  
For a fibration $p:E\to B$ with fibre $F$, Varadarajan \cite{V} established the classical inequality
\begin{equation}\label{eq:Varadarajan-cat-ineq}
    \ct(E)+1\;\leq\;(\ct(F)+1)(\ct(B)+1),
\end{equation}
relating the Lusternik–Schnirelmann category of the total space, the fibre, and the base.  
As a direct application of \Cref{thm: G-fibration}, we obtain the following equivariant analogue.

\begin{proposition}
Let $p:E\to B$ be a $G$-fibration with $B$ $G$-connected. Suppose $E^G\neq \emptyset$ and let $F=p^{-1}(b_0)$, where $b_0=p(e_0)$ for some $e_0\in E^G$. If $F$ is $G$-connected, then
\[
\ct_G(E)+1 \;\leq\; (\ct_G(F)+1)(\ct_G(B)+1).
\]
\end{proposition}

\begin{proof}
Consider diagram \eqref{dia:comm-diagram-G-fibration} with $E'=E$, $B'=B$, $f=id_E$, $g=c_{e_0}$, $\bar f=id_B$, and $\bar g=c_{b_0}$.  
By \Cref{thm: G-fibration},
\[
D_G(id_E,c_{e_0})+1 \;\leq\; \bigl(D_G(id_F,c_{e_0})+1\bigr)\,(\ct_G(B)+1).
\]
Since $D_G(id_E,c_{e_0})=\ct_G(E)$ and $D_G(id_F,c_{e_0})=\ct_G(F)$ (by \Cref{ct-mor-car} and Remark~\ref{rem-ct-mor-car}(1)), the claim follows.
\end{proof}

Turning to topological complexity, Farber and Grant \cite[Lemma~7]{F-G} proved the inequality
\begin{equation}\label{eq:Farber-Grant-inequality}
    \TC(E)+1 \;\leq\; (\TC(F)+1)(\ct(B\times B)+1),
\end{equation}
relating the complexity of the total space and the fibre with the category of $B\times B$.  
In analogy with the equivariant version of Varadarajan’s inequality, we now obtain the following result.

\begin{proposition}
Let $p:E\to B$ be a $G$-fibration with $B$ $G$-connected and $B^G\neq \emptyset$.  
If $F=p^{-1}(b_0)$ for some $b_0\in B^G$, then
\[
\TC_G(E)+1 \;\leq\; (\TC_G(F)+1)(\ct_G(B\times B)+1),
\]
where $G$ acts diagonally on $B\times B$.
\end{proposition}

\begin{proof}
Consider the commutative diagram
\[
\xymatrix{
E\times E \ar[d]_{p \times p} \ar[rr]^{pr_1^E,\,pr_2^E} && E \ar[d]^{p} \\
B\times B \ar[rr]_{pr_1^B,\,pr_2^B} && B,
}
\]
analogous to \eqref{dia:comm-diagram-G-fibration}.  
Since $(B\times B)^H = B^H\times B^H$ for any closed subgroup $H\leq G$, and $B$ is $G$-connected, the product $B\times B$ is also $G$-connected; moreover, $(B\times B)^G = B^G\times B^G \neq \emptyset$.  

Applying \Cref{thm: G-fibration} yields
\[
D_G(pr_1^E,pr_2^E)+1 \;\leq\; \bigl(D_G(pr_1^F,pr_2^F)+1\bigr)(\ct_G(B\times B)+1).
\]
The result follows from \Cref{cor: DG equals TCG}.
\end{proof}

\noindent \textbf{Acknowledgements.}  
The authors gratefully acknowledge the support of the DST–INSPIRE Faculty Fellowship (Faculty Registration No.~IFA24-MA218), Department of Science and Technology, Government of India, and of the Spanish Government under grant PID2023-149804NB-I00.

\bibliographystyle{plain} 
\bibliography{references}

\end{document}